\documentclass{article}
\usepackage{amssymb}
\usepackage{amsthm}
\usepackage[curve,matrix,arrow]{xy}
\theoremstyle{plain}\newtheorem{Theorem}{Theorem}[section]
\theoremstyle{plain}
\theoremstyle{plain}\newtheorem{Corollary}[Theorem]{Corollary}
\theoremstyle{plain}\newtheorem{Lemma}[Theorem]{Lemma}
\theoremstyle{plain}\newtheorem{Proposition}[Theorem]{Proposition}
\theoremstyle{plain}
\theoremstyle{definition}
\theoremstyle{definition}
\theoremstyle{definition}
\theoremstyle{definition}\newtheorem{Remark}[Theorem]{Remark}
\theoremstyle{definition}

    \def\OG{{\mathcal{O}G}}  \def\OGb{{\mathcal{O}Gb}}
    \def\OH{{\mathcal{O}H}}  \def\OHc{{\mathcal{O}Hc}}
    \def\OP{{\mathcal{O}P}}

\def\CO{{\mathcal{O}}}

\def\add{\mathrm{add}}

             \def\ten{\otimes}

\def\Ch{\mathrm{Ch}}
\def\chr{\mathrm{char}}
       
\def\coMack{\mathrm{coMack}}

\def\End{\mathrm{End}}

\def\Hom{\mathrm{Hom}}

\def\Id{\mathrm{Id}}  \def\tenA{\otimes_A} \def\tenE{\otimes_E}
  \def\tenB{\otimes_B} \def\tenF{\otimes_F}

           \def\tenO{\otimes_{\mathcal{O}}}
\def\mod{\mathrm{mod}}

\def\op{\mathrm{op}}

\def\proj{\mathrm{proj}}
           
         \def\tenOQ{\otimes_{\mathcal{O}Q}}
      
           \def\tenOR{\otimes_{\mathcal{O}R}}

             \def\tenOT{\otimes_{\mathcal{O}T}}

\def\ualpha{\underline{\alpha}}

\title{On Morita and derived equivalences for cohomological 
Mackey algebras} 
\author{Markus Linckelmann and Baptiste Rognerud} 
\date{}

\begin{document}

\maketitle

\begin{abstract}
By results of the second author, a source algebra equivalence 
between two $p$-blocks of finite groups induces an equivalence 
between the categories of cohomological Mackey functors associated 
with these blocks, and a splendid derived equivalence between two 
blocks induces a derived equivalence between the corresponding
categories of cohomological Mackey functors. The main result of
this paper proves a partial converse: an equivalence (resp. Rickard 
equivalence) between the categories of cohomological Mackey 
functors of two blocks of finite groups induces a permeable Morita 
(resp. derived) equivalence between the two block algebras.
\end{abstract}

%%%%%%%%%%%%%%%%%%%%%%%%%%%%%%%%%%%%%%%%%%%%%%%%%%%%%%%%%%%%%%%%%%%%
\section{Introduction}

Let $p$ be a prime and $\CO$ a complete local principal ideal domain
with residue field $k=$ $\CO/J(\CO)$ of characteristic $p$; we allow 
the case $\CO=$ $k$. Let $G$ be a finite group.  The blocks of $\OG$ 
are the primitive idempotents in $Z(\OG)$. For $b$ a block of $\OG$,
denote by $\coMack(G,b)$ the abelian category of cohomological Mackey
functors of $G$ associated with $b$, with coefficients in the 
category $\mod(\CO)$ of finitely generated $\CO$-modules. 
The second author showed in \cite{Rogn} that if two block algebras 
$\OGb$ and $\OHc$ of finite groups $G$ and $H$ are 
splendidly Morita or derived equivalent, then
$\coMack(G,b)$ and $\coMack(H,c)$ are equivalent or derived equivalent,
respectively. We show that a Morita or Rickard equivalence between
$\coMack(G,b)$ and $\coMack(H,c)$ induces a Morita or derived 
equivalence between $\OGb$ and $\OHc$.

\begin{Theorem} \label{thm1}
Let $G$, $H$ be finite groups, $b$ a block of $\OG$ and $c$ a
block of $\OH$. An equivalence between the abelian categories 
$\coMack(G,b)$ and $\coMack(H,c)$ (resp. a Rickard equivalence between
their chain homotopy categories) induces a permeable Morita equivalence
(resp. a derived equivalence) between the block algebras $\OGb$ and $\OHc$.
\end{Theorem}

If $\CO$ has characteristic zero, we obtain a converse to
\cite[Proposition 4.5]{Rogn}.

\begin{Corollary} \label{cor1}
Suppose that $\CO$ has characteristic zero. The abelian categories
$\coMack(G,b)$ and $\coMack(H,c)$ are equivalent if and only if $b$ and 
$c$ are splendidly Morita equivalent.
\end{Corollary}

By a result of Scott \cite{Scottnotes} and Puig \cite{Puigbook}, 
two blocks are splendidly Morita equivalent if and only if they have
isomorphic source algebras. 
Theorem \ref{thm1} will be proved in Section \ref{thm1proof} as a 
consequence of the description of the category $\coMack(G,b)$ in terms 
of a source algebra of $b$ in \cite{LincoMack}, and the two theorems 
below. Corollary \ref{cor1} follows from this and Weiss' criterion 
\cite{Weiss}, implying that if $\chr(\CO)=0$, then a permeable Morita 
equivalence is splendid.

\begin{Theorem} \label{thm2}
Let $A$ and $B$ be symmetric $\CO$-algebras. Let $X$ be a finitely
generated $\CO$-free $A$-module and let $Y$ be a finitely
generated $\CO$-free $B$-module. Suppose that $A$ is isomorphic to
a direct summand of $X$ as an $A$-module, and that $B$ is 
isomorphic to a direct summand of $Y$ as a $B$-module. Set
$E=$ $\End_A(X)$ and $F=$ $\End_B(Y)$. 
A Morita equivalence between $E$ and $F$ induces a Morita equivalence 
between $A$ and $B$ which restricts to an equivalence between
$\add(X)$ and $\add(Y)$.
\end{Theorem}

A {\it Rickard equivalence} between two algebras $A$ and $B$ consists
of a  bounded complex $X$ of $A$-$B$-bimodules and a bounded
complex $Y$ of $B$-$A$-bimodules such that the terms of $X$, $Y$ are
finitely generated projective as left and right modules, such that we 
have homotopy equivalences $X\tenB Y \simeq A$ in the homotopy category 
$K(A\tenO A^\op)$ of
finitely generated $A$-$A$-bimodules and $Y\tenA X\simeq$ $B$ in
the corresponding homotopy category $K(B\tenO B^\op)$. A Rickard 
equivalence induces a derived equivalence between $A$ and $B$. 

\begin{Theorem} \label{thm3}
Let $A$ and $B$ be symmetric $\CO$-algebras. Let $X$ be a finitely
generated $\CO$-free $A$-module and let $Y$ be a finitely
generated $\CO$-free $B$-module. Suppose that $A$ is isomorphic to
a direct summand of $X$ as an $A$-module, and that $B$ is 
isomorphic to a direct summand of $Y$ as a $B$-module. Set
$E=$ $\End_A(X)$ and $F=$ $\End_B(Y)$. 
A Rickard equivalence between $E$ and $F$ induces a derived 
equivalence between $A$ and $B$.
\end{Theorem}

Since the categories of cohomological Mackey functors in the statement
of Theorem \ref{thm1} are equivalent to the module categories of the 
corresponding Mackey algebras, the notion of Rickard equivalences 
extends in the obvious way to categories of cohomological Mackey 
functors.
By results of Rickard in \cite{RickDer} and \cite{Ricksplendid}, 
if two symmetric $\CO$-algebras are derived equivalent, then
they are Rickard equivalent, but such a Rickard equivalence may not
be related in an obvious way to a given derived equivalence.
This is essentially the reason why the conclusions in the theorems
above are formulated in terms of derived equivalences rather than 
Rickard equivalences. 

\bigskip\noindent {\bf Notation.}
For $A$ an algebra, we denote by $A^\op$ the opposite algebra.
An $A$-module is a unital left module, unless stated
otherwise. We denote by $\mod(A)$ the category of finitely
generated $A$-modules, and we identify $\mod(A^\op)$ with the category
of finitely generated unital right $A$-modules. 
For $U$ a finitely generated $A$-module, we denote by $\add(U)$ the
full subcategory of $\mod(A)$ consisting of all modules which are 
isomorphic to finite direct sums of summands of $U$. 
We denote by $\Ch(A)$ the category of chain 
complexes of finitely generated $A$-modules, and by $K(A)$ the 
corresponding homotopy category.

\begin{Remark}
For the purpose of proving Theorems \ref{thm2} and \ref{thm3} it would 
be sufficient to require that every projective indecomposable $A$-module 
is isomorphic to a direct summand of $X$, or equivalently, that
the category $\proj(A)$ of finitely generated projective $A$-modules
is contained in the category $\add(X)$. Since $A$ is symmetric, this
condition is equivalent to $X$ having a generator and a cogenerator
as a direct summand. This is the condition which appears for instance
in work of Auslander \cite{Aus}, introducing the notion of representation
dimension, and subsequently in work of Iyama \cite{Iyama},
where the finiteness of the representation dimension of Artin algebras
is proved. 
\end{Remark}

\begin{Remark}
The theme of extending/restricting derived equivalences 
between two finite-dimensional algebras over a field to/from suitable 
idempotent condensations appears in the work of many authors, 
such as  \cite{FaHuKo} and \cite{HuXi}. 
As pointed out by the referee, it would certainly be 
interesting to explore the possibility of common generalisations of 
the results of the present paper and some of the results in the above 
mentioned references. There is some overlap with methods used in 
\cite{MaVi} and \cite{FaHuKo}; we mention this in the Remarks 
\ref{rem-a} and \ref{rem-b} below.
\end{Remark}

It is tempting to speculate, whether the above theorems 
might possibly be of some use towards Brou\'e's abelian defect
conjecture, by playing the conjecture back to a question of
derived equivalences between certain endomorphism algebras with
interesting structural properties. Just to give an example of 
what we have in mind here, one could ask whether - at least
in some special cases - the derived equivalence between blocks 
with abelian defect groups predicted by Brou\'e's abelian defect 
conjecture become more accessible for the quasi-hereditary 
covers (in the sense of Rouquier) of the block algebras. 

%%%%%%%%%%%%%%%%%%%%%%%%%%%%%%%%%%%%%%%%%%%%%%%%%%%%%%%%%%%%%%%%%%%%%
\section{On relatively $\CO$-injective modules}

Let $A$ be an $\CO$-algebra. Suppose that $A$ is free of finite rank
as an $\CO$-module. Let $U$ be a finitely generated left $A$-module.
The module $U$ is called {\it relatively $\CO$-projective} if $U$ is
isomorphic to a direct summand of $A\tenO V$ for some $\CO$-module $V$.
Thus $U$ is projective if and only if $U$ is relatively $\CO$-projective
and $\CO$-free. If $U$ is indecomposable and relatively $\CO$-projective,
then $U$ is isomorphic to a direct summand of either $A$ or $A/\pi^n A$
for some positive integer $n$, because an indecomposable
$\CO$-module is isomorphic to either $\CO$ or $\CO/\pi^n\CO$ for some
positive integer $n$.

Dually, $U$ is called {\it relatively $\CO$-injective}
if $U$ is isomorphic to a direct summand of $\Hom_\CO(A,V)$ for some
$\CO$-module $V$, where the left $A$-module structure on $\Hom_\CO(A,V)$
is given by $(b\cdot\varphi)(a)=$ $\varphi(ab)$ for all $a$, $b\in$ $A$
and $\varphi\in$ $\Hom_\CO(A,V)$. As before, if $U$ is indecomposable
relatively $\CO$-injective, then $U$ is isomorphic to a direct summand
of either $A^*=$ $\Hom_A(A,\CO)$ or of $\Hom_\CO(A,\CO/\pi^n\CO)$ for
some positive integer $n$. Note that for $n=1$ this yields the
injective $k\tenO A$-modules. It is well-known that if $U$ is $\CO$-free,
then $U$ is relatively $\CO$-injective if and only if $k\tenO U$ is 
injective as a $k\tenO A$-module; we include short proofs of this and 
related facts for the convenience of the reader.

\begin{Lemma} \label{lemmaOinj1}
Let $A$ be an $\CO$-algebra which is free of finite rank as an
$\CO$-module, and let $U$ be an $A$-module which is free of finite
rank as an $\CO$-module. Then $U$ is projective if and only if
$k\tenO U$ is a projective $k\tenO A$-module.
\end{Lemma}

\begin{proof}
If $U$ is projective, then $k\tenO U$ is obviously projective as a
$k\tenO A$-module. Suppose conversely that $k\tenO U$ is a projective
$k\tenO A$-module. Then there is a projective $A$-module $P$ such that
$k\tenO P\cong$ $k\tenO U$ as $A$-modules. Since $P$ is projective,
it follows that the obvious map $P\to$ $k\tenO P\cong$ $k\tenO U$
lifts to a map $P\to$ $U$. This map is surjective by Nakayama's
lemma (we use here that $U$ is finitely generated as an $\CO$-module).
Since both $P$ and $U$ are $\CO$-free of the same rank (equal to
the dimension of $k\tenO U$), it follows that the map $P\to$ $U$
obtained in this way is an isomorphism.
\end{proof}

\begin{Lemma} \label{lemmaOinj2}
Let $A$ be an $\CO$-algebra which is free of finite rank as an
$\CO$-module, and let $U$ be an $A$-module which is free of finite
rank as an $\CO$-module. Then $U$ is relatively $\CO$-injective if
and only if the right $A$-module $U^*=$ $\Hom_\CO(U,\CO)$ is projective.
\end{Lemma}

\begin{proof}
Since duality preserves finite direct sums and since $U$ is
projective (resp. relatively $\CO$-injective) if and only if all
direct summands of $U$ have the same property, we may assume that
$U$ is indecomposable. Since also $U$ is $\CO$-free, it follows that
$U$ is relatively  $\CO$-injective if and only if $U$ is isomorphic to
a direct summand of $A^*=$ $\Hom_\CO(A,\CO)$. Applying $\CO$-duality,
this is the case if and only if $U^*$ is isomorphic to a direct summand
of $A^{**}\cong$ $A$ as a right $A$-module, hence if and only if $U^*$
is projective as a right $A$-module.
\end{proof}

\begin{Lemma} \label{lemmaOinj3}
Let $A$ be an $\CO$-algebra which is free of finite rank as an
$\CO$-module, and let $U$ be an $A$-module which is free of finite
rank as an $\CO$-module. Then $U$ is relatively $\CO$-injective if
and only if $k\tenO U$ is an injective $k\tenO A$-module.
\end{Lemma}

\begin{proof}
We may assume that $U$ is indecomposable.
If $U$ is relatively $\CO$-injective, then $U$ is isomorphic to a
direct summand of $A^*=$ $\Hom_\CO(A,\CO)$. Thus $k\tenO U$ is
isomorphic to a direct summand of $(k\tenO A)^*=$ $\Hom_k(k\tenO A,k)$,
and hence $k\tenO U$ is injective as a $k\tenO A$-module.
Suppose conversely that $k\tenO U$ is injective as a $k\tenO A$-module.
Then $k\tenO U$ is isomorphic to a direct summand of the $k$-dual
$(k\tenO A)^*$ of $k\tenO A$. Thus the $k$-dual $(k\tenO U^*)$ of
$k\tenO U$ is isomorphic to a direct summand of the regular right
$k\tenO A$-module $k\tenO A$, hence is projective as a
right $k\tenO A$-module. By the obvious version of Lemma \ref{lemmaOinj1}
for right modules, the $\CO$-dual $U^*$ of $U$ is a projective right
$A$-module. Thus $U$ is relatively $\CO$-injective, by Lemma
\ref{lemmaOinj2}.
\end{proof}

Let $A$ be a finite-dimensional $k$-algebra. The $k$-dual $U^*$ of
a right $A$-module $U$ is then a left $A$-module, and for any
idempotent $e$ in $A$, we have a canonical isomorphism
$(Ue)^*\cong$ $e(U^*)$ of left $eAe$-modules. This isomorphism is
induced by restricting $k$-linear maps $U\to$ $k$ to $Ue$.
It can be regarded as a special case of an adjunction isomorphism: 
since $Ue\cong$ $U\tenA Ae$, we have a natural isomorphism 
$\Hom_k(U\tenA Ae,k)\cong$
$\Hom_A(Ae, \Hom_k(U,k))$. The left side is $(Ue)^*$, and the right
side is $\Hom_A(Ae,U^*)\cong $ $e(U^*)$. Applied to $U=$ $eA$
this yields an isomorphism of left $eAe$-modules $e((eA)^*)\cong$
$(eAe)^*$. We use this in the proof of the following lemma.

\begin{Lemma} \label{AiAinj}
Let $A$ be a finite-dimensional $k$-algebra. Suppose that  for any
primitive idempotent $i$ in $A$ the left $A$-module $Ai$ is injective
if and only if the right $A$-module $iA$ is injective. Let $J$ be a set
of pairwise orthogonal representatives of the conjugacy classes of 
primitive idempotents $j$ with the property that $Aj$ is injective. Set
$e=$ $\sum_{j\in J} j$. Then the $k$-algebra $eAe$ is selfinjective.
\end{Lemma}

\begin{proof}
By the assumptions, the right $A$-module $eA$ is both projective and
injective. Thus its $k$-dual $(eA)^*=$ $\Hom_k(eA,k)$ is
a projective and injective left $A$-module. It follows that any
indecomposable direct summand of $(eA)^*$ is isomorphic to a
direct summand of $Ae$, by the choice of $e$. The indecomposable direct
summands of $Ae$ in any decomposition of $Ae$ are pairwise nonisomorphic.
Similarly for $eA$. Since $Ae$ and $eA$ have
the same number of indecomposable direct factors, it follows that
$(eA)^*\cong$ $Ae$ as left $A$-modules. Multiplying both modules
on the left by $e$ yields an isomorphism of left $eAe$-modules
$e(eA)^*\cong$ $eAe$. The left side is isomorphic to $(eAe)^*$,
and hence $eAe$ is injective as a left $eAe$-module as required.
\end{proof}

\begin{Remark} \label{rem-a}
Lemma \ref{AiAinj} can be proved as a consequence of
\cite[Theorem 1.5]{MaVi}. To see this, it suffices to show
that the class of projective-injective indecomposable $A$-modules
in Lemma \ref{AiAinj} is invariant under the Nakayama
functor $\nu$. If $i$ is a primitive idempotent
in $A$ such that the projective indecomposable $A$-module
$Ai$ is injective, then the right $A$-module $iA\cong$
$\Hom_A(Ai,A)$ is projective and injective by the assumptions
in the Lemma, and hence $\nu(Ai)=$ $(iA)^*$ is projective
and injective.
\end{Remark}

If a finite-dimensional $k$-algebra $A$ has the property that for any
primitive idempotent $i$ the left $A$-module $Ai$ is injective if and
only if the right $A$-module $iA$ is injective, then clearly any block
algebra of $A$ and any algebra Morita equivalent to $A$ inherit the
analogous property. This property applies to the Yoshida type
endomorphism algebras.
This will follow from some general considerations, based on the usual 
translation between direct summands of a module and projective modules 
over its endomorphism algebra.

\begin{Lemma} \label{HomXMinj}
Let $A$ be a finite-dimensional $k$-algebra, $X$ a finite-dimensional
left $A$-module, and set $E=$ $\End_A(X)$. Let $M$ be a direct
summand of $X$.

\smallskip\noindent (i)
If $\Hom_A(X,M)$ is injective as an $E^\op$-module, then
$M$ is injective as an $A$-module.

\smallskip\noindent (ii)
If $\Hom_A(M,X)$ is injective as an $E$-module, then $M$ is projective
as an $A$-module.
\end{Lemma}

\begin{proof}
Suppose that $\Hom_A(X,M)$ is injective as an $E^\op$-module.
Let $\iota : M\to$ $N$ be an injective $A$-homomorphism. We need
to show that $\iota$ is split. The map $\eta : \Hom_A(X,M)\to$
$\Hom_A(X,N)$ sending $\varphi\in$ $\Hom_A(X,M)$ to $\iota\circ\varphi$
is an injective homomorphism of $E^\op$-modules. Thus $\eta$ is
split injective. Therefore there is an $E^\op$-homomorphism
$\epsilon :  \Hom_A(X,N)\to$ $\Hom_A(X,M)$ satisfying
$\epsilon\circ\eta=$ $\Id$, the identity on $\Hom_A(X,M)$.
By the usual general abstract nonsense, $\epsilon$ is induced
by an $A$-homomorphism $\pi : N\to$ $M$; that is, we have
$\epsilon(\psi)=$ $\pi\circ\psi$.  Since $\varphi=$
$\epsilon(\eta(\varphi))=$ $\pi\circ\iota\circ\varphi$ for all
$\varphi\in$ $\Hom_A(X,M)$. Applying this with $\varphi$ a projection
of $X$ onto $M$ implies that $\pi\circ\iota=$ $\Id_M$, and hence 
$\iota$ is split. Thus $M$ is an injective $A$-module. This proves (i).
The proof of (ii) is dual; we sketch the steps. Let $\pi : N\to$ $M$ be
a surjective $A$-homomorphism. Precomposing with $\pi$ induces an
injective homomorphism $\beta : \Hom_A(M,X)\to$ $\Hom_A(N,X)$, which
by the assumptions, is split injective. Any splitting of $\beta$ is
induced by precomposing with an $A$-homomorphism $M\to$ $N$, which is
then shown to be a section of $\pi$. This proves (ii).
\end{proof}

We will use the following elementary fact on tensor products of
finitely generated projective modules.

\begin{Lemma} \label{eMf}
Let $A$ be an $\CO$-algebra which is finitely generated as an
$\CO$-module. Let $e$ be an idempotent in $A$. Suppose that
$U$ is a projective left $A$-module which is a finite direct sum of 
direct summands of $Ae$, or that $V$ is a projective right 
$A$-module which is a finite direct sum of direct summands of $eA$.
The inclusions $eU\subseteq$ $U$ and $Ve\subseteq$ $V$ induce an
isomorphism $Ve\ten_{eAe} eU\cong$ $V\tenA U$.
\end{Lemma}

\begin{proof}
The maps $Ve\ten_{eAe} eU\to$ $V\tenA U$ induced by the inclusions
$eU\subseteq$ $U$ and $Ve\subseteq$ $V$ are a natural transformation
from the bifunctor $(U,V)\mapsto$ $Ve\ten_{eAe} eU$ to the
bifunctor $(U,V)\mapsto$ $V\tenA U$. This natural transformation is
$\CO$-linear in both arguments, so it suffices to show that it
yields an isomorphism if $U=Ae$ or if $V=eA$. If $U=Ae$, then
$Ve \ten_{eAe} eU =$ $Ve \ten_{eAe} eAe \cong$ $Ve\cong$
$V\tenA Ae=$ $V\tenA U$. A similar argument shows that if $V=$ $eA$,
then we have an isomorphism $Ve\ten_{eAe} eU\cong$ $V\tenA U$,
proving the statement.
\end{proof}

\begin{Lemma} \label{contractible-lift}
Let $A$ be an $\CO$-algebra which is finitely generated as an
$\CO$-module.
Let $U$ be a bounded complex of finitely generated projective
$A$-modules. 
Suppose that $k\tenO U$ has a contractible direct summand $W$ as
a complex of $k\tenO A$-modules. Then $U$ has a contractible
direct summand $V$ such that $k\tenO V=$ $W$.
\end{Lemma}

\begin{proof}
Set $\bar A=$ $k\tenO A$ and $\bar U=$ $k\tenO U$.
Since $U$ is bounded, the algebra $\End_{\Ch(A)}(U)$ is finitely
generated as an $\CO$-module, and $\End_{\Ch(\bar A)}(\bar U)$ is
finite-dimensional. 
Denote by $C$ the ideal in $\End_{\Ch(A)}(U)$ of
chain maps $\psi : U\to$ $U$ which satisfy $\psi\sim 0$; that is, 
$C$ is the kernel of the canonical algebra homomorphism
$\End_{\Ch(A)}(U)\to$ $\End_{K(A)}(U)$. Similarly, denote by
$D$ the kernel of the canonical algebra homomorphism
$\End_{\Ch(\bar A)}(\bar U)\to$ $\End_{K(\bar A)}(\bar U)$. 
Since the components of $U$ are projective, any
homotopy $\bar U\to$ $\bar U[-1]$ lifts to a homotopy 
$U\to$ $U[-1]$. It follows that the canonical map $C\to$ $D$
is surjective. The summand  $W$ of $\bar U$ corresponds to an
idempotent $\eta$ in $\End_{\Ch(\bar A)}(\bar U)$. Since $W$ is 
contractible, this idempotent is contained in $D$. Thus, by standard lifting 
theorems, there is an idempotent $\hat\eta$ in $C$ which lifts 
$\eta$. Thus $V=$ $\hat\eta(U)$ is a contractible direct summand
of $U$ which lifts $W$.  
\end{proof}

The following result identifies bounded complexes of finitely
generated modules which are both projective and injective. 

\begin{Lemma} \label{projinj-bounded}
Let $A$ be a finite-dimensional $k$-algebra. Let $U$ be a bounded
chain complex of finitely generated projective $A$-modules. Suppose 
that $U$ has no nonzero contractible direct summand as a chain complex, 
and that for any bounded above acyclic chain complex $C$ in $\Ch(A)$ we 
have $\Hom_{K(A)}(C,U)=\{0\}$. Then the components of $U$ are injective.
\end{Lemma}

\begin{proof}
Let $\alpha : U\to$ $I_U$ be an injective resolution of $U$; that is,
$I_U$ is a bounded above chain complex of finitely generated
injective $A$-modules and $\alpha$ is a quasi-isomorphism.
Denote by $\ualpha$ the image of $\alpha$ in $\Hom_{K(A)}(U,I_U)$.
Consider the associated exact triangle in $K(A)$,
$$\xymatrix{ U \ar[r]^{\ualpha}& I_U\ar[r] & C(\alpha) \ar[r]&
U[1]}$$
Since $\alpha$ is a quasi-isomorphism, it follows its cone $C(\alpha)$ 
is acyclic. By the assumptions on $U$, the morphism $C(\alpha)\to$ 
$U[1]$ in this triangle is zero. Therefore, by a standard property of 
triangulated categories (see e. g. \cite[Lemma 3.4.9]{Zim}),
the morphism $\ualpha$ is a split monomorphism in $K(A)$. That is,
there exists a chain map $\delta : I_U\to$ $U$ such that 
$\delta\circ\alpha\sim$ $\Id_U$. 

Since $U$ is bounded, the algebra
$\End_{\Ch(A)}(U)$ is finite-dimensional. We use as before the fact
that idempotents in this algebra correspond to direct summands of $U$ 
as a chain complex, and that contractible direct summand correspond to 
those idempotents which are in the kernel of the canonical algebra 
homomorphism $\End_{\Ch(A)}(U)\to$ $\End_{K(A)}(U)$. 

By the assumptions on $U$, this kernel contains no
idempotents, and hence is contained in the radical $J(\End_{\Ch(A)}(U))$.
Since $\delta\circ\alpha$ maps to the identity in $\End_{K(A)}(U)$,
it follows that $\delta\circ\alpha$ is invertible in $\End_{\Ch(A)}(U)$.
Thus the chain map $\beta=$ $(\delta\circ\alpha)^{-1}\circ\delta$
satisfies $\beta\circ\alpha=$ $\Id_U$. This shows that $U$ is
isomorphic to a direct summand of $I_U$. In particular, the components
of $U$ are injective.
\end{proof}

%%%%%%%%%%%%%%%%%%%%%%%%%%%%%%%%%%%%%%%%%%%%%%%%%%%%%%%%%%%%%%%%%%
\section{Proof of Theorems \ref{thm2} and \ref{thm3}}

An $\CO$-algebra $A$ is {\it symmetric} if $A$ is free of finite rank as
an $\CO$-module, and if $A$ is isomorphic to its $\CO$-dual $A^*$ as
an $A$-$A$-bimodule. One of the special features of a symmetric
$\CO$-algebra $A$ is that the two duality functors with respect to
$A$ and $\CO$ are isomorphic; that is, for any left $A$-module $U$, there
is an isomorphism $\Hom_A(U,A)\cong$ $U^*$ of right $A$-modules which
is natural in $U$. More precisely, any choice of a bimodule isomorphism
$A\cong$ $A^*$ induces such an isomorphism of duality functors as
follows: if $s\in$ $A^*$ is the image of $1$ under a bimodule
isomorphism $A\cong$ $A^*$, then the map sending $\varphi\in$
$\Hom_A(U,A)$ to $s\circ\varphi\in$ $U^*$ is an isomorphism, for any
$A$-module $U$. The naturality implies in particular that this
isomorphism is an isomorphism as right $\End_A(U)$-modules.

\begin{Lemma} \label{EXproj}
Let $A$ be a finite-dimensional $k$-algebra and let $X$ be a
finite-dimensional $A$-module. Set $E=$ $\End_A(X)$.
If $X$ has a direct summand isomorphic to $A$ as an $A$-module, then
$X$ is projective as an $E$-module. If in addition $A$ is symmetric,
then $X^*$ is projective as an $E^\op$-module.
\end{Lemma}

\begin{proof}
If $A$ is isomorphic to a direct summand of $X$, then $\Hom_A(A,X)$
is a projective $E$-module, and clearly $\Hom_A(A,X)\cong$ $X$.
If in addition $A$ is symmetric, then we have a natural isomorphism
$X^*\cong$ $\Hom_A(X,A)$, and this is a projective $E^\op$-module.
\end{proof}

\begin{Proposition} \label{EndXinj}
Let $A$ be a symmetric $k$-algebra and $X$ a finite-dimensional
$A$-module. Suppose that $A$ is isomorphic to a direct summand of $X$.
Set $E=$ $\End_A(X)$. Let $U$ be a direct summand of $X$.
The following are equivalent.

\smallskip\noindent (i)
The $A$-module $U$ is projective.

\smallskip\noindent (ii)
The $E^\op$-module $\Hom_A(X,U)$ is injective.

\smallskip\noindent (iii)
The $E$-module $\Hom_A(U,X)$ is injective.
\end{Proposition}

\begin{proof}
If $\Hom_A(X,U)$ is an injective $E^\op$-module, then by Lemma
\ref{HomXMinj} (i) the $A$-module $U$ is injective, hence also 
projective, since we assume that $A$ is symmetric. Thus (ii) implies
(i). Similarly, if $\Hom_A(U,X)$ is injective, then by Lemma 
\ref{HomXMinj} (ii) the $A$-module $U$ is projective. Thus (iii) 
implies (i). In order to show that (i) implies (ii) and (iii), it 
suffices to show that $\Hom_A(X,A)$ is
an injective $E^\op$-module and that $\Hom_A(A,X)$ is an injective
$E$-module. Thus it suffices to show that their duals $\Hom_A(X,A)^*$
and $\Hom_A(A,X)^*$ are projective as modules over $E$ and $E^\op$,
respectively. Since $A$ is symmetric, we have a natural
isomorphism $\Hom_A(X,A)\cong \Hom_k(X,k)=$ $X^*$.
The naturality implies in particular, that these isomorphisms are
isomorphisms as $E^\op$-modules.
Thus we have an isomorphism of $E$-modules $\Hom_A(X,A)^*\cong$ $X$,
and this is a projective $E$-module by Lemma \ref{EXproj}.
Similarly, note that $\Hom_A(A,X)\cong$ $X$, and
hence that $\Hom_A(A,X)^*\cong X^*$, which is indeed
a projective $E^\op$-module, again by Lemma \ref{EXproj}.
\end{proof}

\begin{proof}[{Proof of Theorem \ref{thm2}}]
We use the notation of Theorem \ref{thm2}.
By Proposition \ref{EndXinj} and Lemma \ref{lemmaOinj3}, projective
indecomposable modules over $E$ and $F$ which are also relatively
$\CO$-injective, correspond to the indecomposable summands of $A$ and $B$,
respectively.
Denote by $e$ an idempotent in $E$ such that $e(X)\cong$ $A$, and
denote by $f$ an idempotent in $F$ such that $f(Y)\cong$ $B$.
Then $Ee$ is a direct sum of indecomposable $E$-modules which
are projective and relatively $\CO$-injective, and any
indecomposable $E$-module which is projective and relatively
$\CO$-injective is isomorphic to a direct summand of $Ee$.
The right $E$-module $eE$, the left $F$-module $Ff$ and the right
$F$-module $fF$ have the analogous properties.

Let $M$ be an $E$-$F$-module and $N$ an $F$-$E$-bimodule inducing a
Morita equivalence; that is, $M$, $N$ are finitely generated
projective as left and right modules, and we have bimodule
isomorphisms $M\tenF N\cong$ $E$ and $N\tenE M\cong$ $F$.
A Morita equivalence between $E$ and $F$ preserves projective
indecomposables which are also relatively $\CO$-injective.
Thus $N\tenE Ee\cong$ $Ne$ is a direct sum of summands of
$Ff$. In particular, $fNe$ is projective as a $fFf$-module. 
Similarly, $eM$, as a right $F$-module, is a direct
sum of summands of $fF$, and hence $eMf$ is projective as a right
$fFf$-module. Using Lemma \ref{eMf}, applied to $E$ and $F$ instead of $A$ 
and $B$, we have isomorphisms
$$eEe \cong eM\tenF Ne\cong eMf\ten_{fFf} fNe$$
Exchanging the roles of $E$ and $F$ shows similary that
$fFf\cong$ $fNe\ten_{eEe} eMf$, and that $eMf$ and $fNe$ are both
projective as left and as right modules. Thus
the bimodules $eMf$ and $fNe$ induce a Morita equivalence between 
$eEe$ and $fFf$. Since $eEe\cong$ $\End_A(e(X))\cong$
$\End_A(A)\cong$ $A^\op$ and similarly $fFf\cong$ $B^\op$, passing to
opposite algebras yields a Morita equivalence between $A$ and $B$.

We need to show that this Morita equivalence restricts to an 
equivalence between $\add(X)$ and $\add(Y)$. It suffices to 
show that the equivalence $\mod(A)\cong$ $\mod(B)$ sends
$\add(X)$ to $\add(Y)$. Since $A$ and $B$ are symmetric, it
suffices to show that $fNe\ten_{eEe}-$ sends $\add(X^*)$ to
$\add(Y^*)$. Let $V$ be an $A$-module in $\add(X)$. Then 
$\Hom_A(V,X)$ is a projective $E$-module. Since $N\ten_E-$ is
an equivalence of categories, it follows that $N\ten_E\Hom_A(V,X)$
is a projective $F$-module. Thus there is a $B$-module $W$ in
$\add(Y)$ such that $N\ten_E\Hom_A(V,X)\cong$ $\Hom_B(W,Y)$
as $F$-modules. Multiplying by $f$ yields an isomorphism of
$fFf$-modules 
$$fN\ten_E\Hom_A(V,X)\cong f\Hom_B(W,Y)\cong \Hom_B(W,B)$$
Since $fN$ is a direct sum of summands of $eE$, it follows from
Lemma \ref{eMf} that the left side is isomorphic to
$$fNe\ten_{eEe} e\Hom_A(V,X)\cong fMe\ten_{eEe}\Hom_A(V,A)$$
Since $A$ and $B$ are symmetric, we have $\Hom_A(V,A)\cong$
$V^*$ and $\Hom_B(W,B)\cong$ $W^*$. This shows that
$fNe\ten_{eEe} V^*\cong$ $W^*$, and hence the functor
$fNe\ten_{eEe}-$ sends $\add(X^*)$ to $\add(Y^*)$ as claimed.
\end{proof}

\begin{proof}[{Proof of Theorem \ref{thm3}}]
We use the notation of Theorem \ref{thm3}.
Let $M$ be a bounded complex of $E$-$F$-bimodules and $N$ a bounded
complex of $F$-$E$-bimodules inducing a Rickard equivalence; that is, the
components of $M$, $N$ are finitely generated
projective as left and right modules, and we have homotopy equivalences
of chain complexes of bimodules $E$-$E$-bimodules $M\tenF N\simeq$ $E$ and 
$F$-$F$-bimodules $N\tenE M\simeq$ $F$.
Denote by $e$ an idempotent in $E$ such that $e(X)\cong$ $A$, and
denote by $f$ an idempotent in $F$ such that $f(Y)\cong$ $B$.
Multiplying the above homotopy equivalences by $e$ and $f$ on both
sides yields homotopy equivalences of chain complexes of
$eEe$-$eEe$-bimodules $eM\tenF Ne\simeq eEe$ and of
$fFf$-$fFf$-bimodules $fN\tenE Mf\simeq$ $fFf$.

We will show that there are quasi-isomorphisms 
$$eMf\ten_{fFf}fNe\to eM\tenF Ne$$ 
$$fNe\ten_{eEe} eMf\to fN\ten_E Mf$$ 
whose restriction to the left and to the right are homotopy 
equivalences, and we will then see that this implies that the functors 
$eMf\ten_{fFf}-$ and $fNe\ten_{eEe}-$ induce inverse derived 
equivalences. 
 
The complex $Ne\cong$ $N\tenE Ee$ is a bounded complex of finitely 
generated projective $F$-modules. We will show that 
$$Ne \cong N_0 \oplus N_1$$
for some contractible complex $N_1$ and a complex $N_0$ whose 
components consist of finite direct sums of summands of $Ff$. In order 
to show this,  by Lemma \ref{contractible-lift}, we may assume that 
$\CO=k$. Since $Ee$ is injective, we have $\Hom_{K(E)}(C,Ee)=$ $\{0\}$ 
for any acyclic bounded above complex $C$ of finitely generated 
$E$-modules. Since $N\ten_E-$ induces an equivalence of homotopy
categories of chain complexes $K(E)\cong$ $K(F)$, it follows that we 
have $\Hom_{K(F)}(D, Ne)=$ $\{0\}$ for any 
acyclic bounded above complex $D$ of finitely generated $F$-modules. It 
follows from Lemma \ref{projinj-bounded}, that the indecomposable 
direct summands of $Ne$ which are not contractible consist of injective 
$F$-modules, hence of sums of summands of $Ff$. Reverting to general 
$\CO$, multiplying the previous isomorphism by $f$ on the left yields an 
isomorphism of chain complexes of $fFf$-modules
$$fNe \cong fN_0 \oplus fN_1$$
such that $fN_0$ is a bounded complex of finitely generated projective
$fFf$-modules, and $fN_1$ is a bounded contractible complex.
The same argument shows  that we have an isomorphism of chain
complexes of right $F$-modules
$$eM\cong M_0\oplus M_1$$
where $M_0$ is a complex of right $F$-modules which are
finite direct sums of summands of $fF$, and $M_1$
is a contractible complex of right $F$-modules. Thus, as before,
multiplying this isomorphism on the right by $f$ yields an isomorphism of 
chain complexes of right $fFf$-modules
$$eMf\cong M_0f \oplus M_1f$$
such that $M_0f$ is a bounded complex of finitely generated projective
right $fFf$-modules and $M_1f$ is a bounded contractible complex.
Thus we have a decomposition as complexes of right $eEe$-modules
$$eM\ten_F Ne = M_0\tenF Ne \oplus M_1 \tenF Ne \cong$$
$$ M_0f \ten_{fFf} fNe \oplus M_1\tenF Ne$$
where we have made use of Lemma \ref{eMf} for the isomorphism.
We also have
$$eMf\ten_{fFf} fNe = M_0f \ten_{fFf} fNe \oplus M_1f\ten_{fFf} fNe$$
In both of these isomorphisms, the right most terms are contractible
as chain complexes of right $eEe$-modules, because $M_1f$ is 
contractible as a chain complex of right $fFf$-modules.
Thus we have a chain map of complexes of $eEe$-$eEe$-bimodules
$$eMf\ten_{fFf} fNe \to eM\tenF Ne\simeq eEe$$
which restricts to a homotopy equivalence as a chain map of complexes
of right $eEe$-modules, and, by the analogous argument, restricts
to a homotopy equivalence as a chain map of complexes of left
$eEe$-modules. In particular, this bimodule chain map is a
quasi-isomorphism. Similarly, we have a bimodule quasi-isomorphism
$$fNe\ten_{eEe} eMf \to fN\ten_E Mf \simeq fFf$$
which restricts to homotopy equivalences on the left and on the right.

We show next that the functor $eMf\ten_{fFf}-$ from $\Ch(fFf)$ to
$\Ch(eEe)$ preserves quasi-isomorphisms. We use the right $fFf$-chain
complex decomposition $eMf=$ $M_0f\oplus M_1f$ above.
If $\beta : V\to V'$ is a quasi-isomorphism in $\Ch(fFf)$, then
$\Id_{eMf}\ten\beta : eMf\ten_{fFf} V\to$ $eMf\ten_{fFf} V'$
decomposes as a direct sum of chain maps of complexes of
$\CO$-modules $M_0f\ten_{fFf} V\to$ $M_0f\ten_{fFf} V'$ and
$M_1f\ten_{fFf} V\to$ $M_1f\ten_{fFf} V'$. The first of these
is a quasi-isomorphism because $M_0f$ is a bounded complex of
projective right $fFf$-modules. The second is trivially a 
quasi-isomorphism, since both 
$M_1f\ten_{fFf} V$ and $M_1f\ten_{fFf} V'$ are acyclic, as $M_1f$ is
contractible as a complex of right $fFf$-modules. 
Thus the functor $eMf\ten_{fFf}-$ induces a functor on derived
categories; similarly for $fNe\ten_{eEe}-$. These two functors
are inverse to each other as functors on the derived categories. 
Indeed, since the above bimodule chain map 
$eMf\ten_{fFf} fNe \to eEe$ is a homotopy equivalence as chain map
of complexes of right $eEe$-modules, it follows that for any
complex $U$ in $\Ch(eEe)$, the induced chain map
$eMf\ten_{fFf} fNe\ten_{eEe} U \to U$ is a homotopy equivalence as 
a chain map of complexes of $\CO$-modules, hence a quasi-isomorphism
as a chain map of complexes of $eEe$-modules.
The result follows.
\end{proof}

\begin{Remark} \label{rem-b}
The above proof does not show that the quasi-isomorphisms 
$eMf\ten_{fFf}fNe\to$ $eM\tenF Ne$ and $fNe\ten_{eEe} eMf\to$ 
$fN\ten_E Mf$ are homotopy equivalences as bimodule chain maps.
It also does not show that $eMf$ and $fNe$ are projective as
complexes of left and right modules. In particular, this proof
does not show that $eMf$ and $fNe$ are Rickard complexes, and it seems
unclear whether the induced derived equivalence preserves the
subcategories of chain complexes over $\add(X)$ and $\add(Y)$. 
The proof does show that the derived equivalence induced
by $M$ and $N$ restricts to an equivalence 
$$K^b(\add(Ee))\cong K^b(\add(Ff))\ ,$$ 
thanks to the isomorphism $Ne \cong N_0 \oplus N_1$ for some 
contractible complex $N_1$ and a complex $N_0$ with components in 
$\add(Ff)$, and the analogous isomorphism for $Mf$, where the 
notation is as in the proof above. In the case where the base
ring $\CO$ is a field, this can also be deduced from 
\cite[Theorem 4.3]{FaHuKo}, making use of Remark \ref{rem-a}.  
\end{Remark}

%%%%%%%%%%%%%%%%%%%%%%%%%%%%%%%%%%%%%%%%%%%%%%%%%%%%%%%%%%%%%%%%%%
\section{Proof of Theorem \ref{thm1} and Corollary \ref{cor1}}
\label{thm1proof}

The proof of Theorem \ref{thm1} is played back to the theorems 
\ref{thm2} and \ref{thm3}, together with
description of cohomological Mackey functors in terms of
source algebras of blocks in \cite{LincoMack}, extending ideas going
back to Yoshida \cite{YoshidaII}. 

\begin{Proposition} \label{source1}
Let $A$ be a source algebra of a block of a finite group with
defect group $P$. Set $X= \oplus_Q\ A\tenOQ \CO$ and $E=$ $\End_A(X)$,
where in the direct sum $Q$ runs over the subgroups of $P$. For $\iota$
a primitive idempotent in $E$, the following are equivalent.

\smallskip\noindent (i)
$E\circ \iota$ is a relatively $\CO$-injective left $E$-module.

\smallskip\noindent (ii)
$\iota \circ E$ is a relatively $\CO$-injective right $E$-module.

\smallskip\noindent (iii)
$\iota(X)$ is isomorphic to a projective indecomposable left
$A$-module.
\end{Proposition}

\begin{proof}
Note that $A$ is a symmetric $\CO$-algebra and that $A$ is isomorphic
to the summand indexed by the trivial group $1$ in the direct
sum $X=\oplus_Q\ A\tenOQ \CO$. Thus the result follows from Proposition
\ref{EndXinj}, combined with Lemma \ref{lemmaOinj3}.
\end{proof}

\begin{proof}[{Proof of Theorem \ref{thm1}}]
Denote by $P$ a defect group and by $A$ a source algebra of the block
$b$ of $\OG$. Similarly, denote by $Q$ a defect group and by $B$ a source
algebra of the block $c$ of $\OH$. 
Set $X= \oplus_R\ A\tenOR \CO$, where $R$ runs over the subgroups of $P$,
and set $E=$ $\End_A(X)$. Similarly, 
$Y= \oplus_T\ A\tenOT \CO$, where $T$ runs over the subgroups of $Q$, 
and set $F=$ $\End_B(Y)$. It follows from Theorem \ref{thm2} 
(resp. Theorem \ref{thm3}) that if
$E$ and $F$ are Morita equivalent (resp. Rickard equivalent) then 
$A$ and $B$ are permeable Morita equivalent (resp. derived equivalent). 
By \cite[Theorem 1.1]{LincoMack} we have
$\coMack(G,b)\cong$ $\mod(E^\op)$ and $\coMack(H,c)\cong$
$\mod(F^\op)$, whence the result.
\end{proof}

A permeable Morita equivalence between block algebras over $k$ need
not be splendid; see \cite[Remark 4.7]{Rogn}. In characteristic
zero, however, permeable Morita equivalences are splendid.

\begin{Proposition} \label{permeablezero}
If $\CO$ has characteristic zero, then a permeable Morita equivalence
between two blocks of finite group algebras over $\CO$ is splendid. 
\end{Proposition}

\begin{proof}
Let $B$, $B'$ be block algebras of some finite group algebras over 
$\CO$, with defect groups $P$, $P'$, respectively. Let $M$ be a 
$B$-$B'$-bimodule inducing a permeable Morita equivalence. Then 
$M\ten_{B'}-$ sends the $p$-permutation $B'$-module 
$B'\ten_{\CO P'}\CO$ to the  $p$-permutation $B$-module 
$M\ten_{\CO P'}\CO$. In particular, $M$ is free as a left $\OP$-module, 
as a right $\CO P'$-module, and $M\ten_{\CO P'}\CO$ is a permutation 
$\OP$-module. Note that since $M$ is free as a right $\CO P'$-module, 
the quotient $M\ten_{\CO P'}\CO$ is isomorphic to the fixpoints 
$M^{P'}$ in $M$ with respect to the right action of $P'$ on $M$. 
Weiss' criterion \cite[Theorem 2]{Weiss} (adapted to the more general 
coefficient ring $\CO$ in \cite[Appendix 1]{Puigbook}),  
applied to $P\times P'$ and the normal subgroup $1\times P'$ instead of 
$G$ and $N$, respectively, implies that $M$ is a permutation 
$\CO(P\times P')$-module. 
\end{proof}

\begin{proof}[{Proof of Corollary \ref{cor1}}]
If $B$ and $B'$ are splendidly Morita equivalent, then their associated 
categories of cohomological Mackey functors are equivalent by 
\cite[Proposition 4.5]{Rogn}. 
Conversely, if $\chr(\CO)=0$ and if the categories of cohomological 
Mackey functors of $B$, $B'$ are equivalent, then $B$, $B'$ are 
permeable Morita equivalent by Theorem \ref{thm1}, hence splendidly
Morita equivalent by Proposition \ref{permeablezero}.
\end{proof}
 
%%%%%%%%%%%%%%%%%%%%%%%%%%%%%%%%%%%%%%%%%%%%%%%%%%%%%%%%%%%%%%%%%%
\section{A remark on nilpotent blocks}

By results of Puig in \cite{Punil} and \cite{Puigbook}, if a block $b$ 
of a finite group algebra $\OG$ is nilpotent, then $\OGb$ is Morita 
equivalent to $\OP$, and if $\CO$ has characteristic zero, then the 
converse holds as well. Thus, if $\chr(\CO)=0$, then
Corollary \ref{cor1} implies that $b$ is nilpotent with a source algebra
isomorphic to $\OP$ if and only if $\coMack(G,B)\cong\coMack(P)$.
For derived equivalences, Theorem \ref{thm1} this yields the following.

\begin{Theorem}
Let $G$ be a finite group and $b$ a block of $\OG$. Suppose that
$\CO$ has characteristic zero. If the categories $\coMack(G,b)$ and 
$\coMack(P)$ are Rickard equivalent, then $b$ is nilpotent, with defect 
groups isomorphic to $P$.
\end{Theorem}

\begin{proof}
Suppose that the categories $\coMack(G,b)$ and $\coMack(P)$ are 
Rickard equivalent. Then, by Theorem \ref{thm1}, the algebras
$\OGb$ and $\OP$ are  derived equivalent. Since $\OP$ is split
local, it follows from \cite[6.7.5]{Zim} that $\OGb$ and $\OP$
are Morita equivalent. Thus, by \cite[Theorem 8.2]{Puigbook}, $b$ is 
nilpotent with defect groups isomorphic to $P$.
\end{proof}

%One needs to investigate the converse - which is perhaps not true.
%Here is where one should look for a counter example: if $b$ is 
%nilpotent with a quaternion defect group $Q_8$ and with endopermutation 
%source $V$ of dimension $3$ of the simple modules, then $V$ does not 
%have an endo-split $p$-permutation resolution, hence $\OGb$ is not 
%splendidly equivalent to $\OQ_8$ (although these two algebras are Morita
%equivalent).  

\bigskip\noindent
{\it Acknowledgements.} The first author would like to 
acknowledge support from the EPSRC grant EP/M02525X/1. Both 
authors would like to thank the EPFL for its hospitality
during the special research semester on Local Representation 
Theory and Simple Groups in 2016.

\bigskip
%%%%%%%%%%%%%%%%%%%%%%%%%%%%%%%%%%%%%%%%%%%%%%%%%%%%%%%%%%%%%%%%%%%%%

\end{document}